\documentclass{article}
\usepackage[latin2]{inputenc}
\usepackage{amsmath, amssymb, amsthm}
\newtheorem{theo}{Theorem}

\newtheorem{cor}{Corollary}
\newtheorem{megj}{Remark}
\newtheorem{ex}{Examples}

\title{Inequalities for multiplicative arithmetic functions}
\author{József Sándor}
\date{Babe\c{s}--Bolyai University\\Department of Mathematics\\Str. Kog\u{a}lniceanu Nr. 1\\400084 Cluj--Napoca, Romania\\\texttt{jsandor@math.ubbcluj.ro, jjsandor@hotmail.com}}

\begin{document}
\maketitle
\begin{abstract}
Let $f,g:\mathbb{N}\to \mathbb{R}$ be arithmetic functions. The function $f$ is called multiplicative, if $f(mn)=f(m)f(n)$ for all $(m,n)=1;$ and sub (super)-multiplicative, if $f(mn)\substack{\leq\\ (\geq)}= f(m)\cdot f(n)$ for all $m,n\geq 1.$ The function $g$ is called sub (super)-homogeneous, if $f(mn)\substack{\leq\\ (\geq)}mf(n)$ for all $m,n\geq 1.$ We say that $f$ is 2-sub (super)-multiplicative, if $\big(f(mn)\big)^2\leq f(m^2)f(n^2)$ for all $m,n\geq 1;$ and 2-sub (super)-homogeneous if $\big(f(mn)\big)^2\leq m^2f(n^2)$ for all $m,n\geq 1.$ In this paper we study the above classes of functions, with applications.
\end{abstract}



\section{Introduction}

Let $\varphi,d,\sigma$ denote, as usual, the classical arithmetic functions, representing Euler's totient, the number of divisors, and the sum of divisors, respectively. Then it is well-known (see e.g. \cite{1}, \cite{7}) that $\varphi(1)=d(1)=\sigma(1)=1,$ and all these functions are multiplicative, i.e. satisfy the functional equation
\begin{equation}\label{eq:1}
f(mn)=f(m)f(n) \mbox{ for all } (m,n)=1.
\end{equation}

Here $f:\mathbb{N}\to \mathbb{R},$ as one can find also many examples to \eqref{eq:1}, when $f$ has not integer values, namely e.g. $f(n)=\cfrac{\sigma(n)}{\varphi(n)},$ or $f(n)=\cfrac{\sigma(n)}{d(n)},$ etc.

The functions $d$ and $\sigma$ are classical examples of the so-called \emph{"sub-multiplicative"} function, i.e. satisfying
\begin{equation}\label{eq:2}
f(mn)\leq f(m)f(n) \mbox{ for all } m,n\geq 1.
\end{equation}

The function $\varphi$ is \emph{"super-multiplicative",} i.e. satisfies the inequality \eqref{eq:2} in reversed order:
\begin{equation}\label{eq:3}
f(mn)\geq f(m)f(n) \mbox{ for all } m,n\geq 1.
\end{equation}

Similarly, we will say that $g$ is \emph{"sub-homogeneous",} when
\begin{equation}\label{eq:4}
g(mn)\leq mg(n) \mbox{ for all } m,n\geq 1,
\end{equation}
 and \emph{"super-homogeneous",} whenever
\begin{equation}\label{eq:5}
g(mn)\geq mg(n) \mbox{ for all } m,n\geq 1.
\end{equation}

For example, the function $\varphi$ satisfies inequality \eqref{eq:4}, while $\sigma,$ the inequality \eqref{eq:5}. As the function $d$ satisfies inequality \eqref{eq:2}, and $d(n)\leq n$ for all $n,$ clearly $g=d$ is also a sub-homogeneous function.

Let $k\geq 2$ be a positive integer. In what follows, we will say that $f$ is \emph{"k-sub-multiplicative"} function, if one has
\begin{equation}\label{eq:6}
\big(f(mn)\big)^k\leq f(m^k)f(n^k) \mbox{ for all } m,n\geq 1.
\end{equation}

For example, in \cite{3} it is proved that $\varphi$ is 2-sub-multiplicative, while in \cite{4} that $d$ and $\sigma$ are \emph{"2-super-multiplicative,"} to the effect that they satisfy the functional inequality
\begin{equation}\label{eq:7}
\big(f(mn)\big)^k\geq f(m^k)f(n^k) \mbox{ for all } m,n\geq 1,
\end{equation}
when $k=2.$ In fact, as we shall see, $\varphi$ satisfies \eqref{eq:6} for any $k\geq 2,$ and similarly $d$ and $\sigma$ satisfy relation \eqref{eq:7}.

In papers \cite{4}, among the inequalities \eqref{eq:2}--\eqref{eq:7} for the particular functions $\varphi,\sigma,d$ and $k=2,$ we have shown also e.g. that  $f(n)=\cfrac{\sigma(n)}{d(n)}$ is super-multiplicative (i.e., satisfies \eqref{eq:3}), 2-sub-multiplicative (i.e. satisfies \eqref{eq:6} for $k=2$), and it is sub-homogeneous (i.e. satisfies \eqref{eq:4}).

A last class of functions, which we will introduce here is the class of \emph{"k-sub-homogeneous" ("k-super-homogeneous")} functions, i.e. functions $g$ with the property
\begin{equation}\label{eq:8}
\big(g(mn)\big)^k\substack{\leq\\ (\geq)} m^k g(n^k), \quad m,n\geq 1;
\end{equation}
where $k\geq 2$ is a fixed positive integer.

\section{Main result}

The first result is almost trivial, we state it for the sake of completeness and for applications in the next sections.

\begin{theo}
Let $f,g$ be two nonnegative arithmetic functions.
\begin{enumerate}
  \item [a)] If $f$ and $g$ are sub-mult (abbreviation for "sub-multiplicative"), then \mbox{$f\cdot g$} is sub-mult, too. Similarly, if $f$ and $g$ are sup-mult (i.e. "super-multiplicative"), then $f\cdot g$ is sup-mult, too.
  \item [b)] If $f>0$ is sub-mult (sup-mult), then $\cfrac{1}{f}$ is sup-mult (sub-mult).
  \item [c)] If $f$ is sup-mult (sub-mult) and $g>0$ sub-mult (sup-mult), then  $\cfrac{f}{g}$ is sup-mult (sub-mult).
  \item [d)] If $f$ and $g$ are sub-mult, then $f+g$ is sub-mult, too.
\end{enumerate}
\end{theo}

\begin{proof}
a) -- c) are easy consequences of the definitions \eqref{eq:2} and \eqref{eq:3}, and the positivity of functions; for d) put $h(n)=f(n)+g(n).$ Then \begin{align*}
&h(nm)=f(nm)+g(nm)\leq f(n)f(m)+g(n)g(m)\leq \\
& \leq  h(n)h(m)=[f(n)+g(n)]\cdot [f(m)+g(m)]=\\
& = f(n)f(m)+g(n)g(m)+A,
\end{align*} where $A=f(n)g(m)+g(n)f(m)\geq 0.$
\end{proof}

\begin{ex}
\begin{enumerate}
\item [a)] $f(n)=\cfrac{\sigma(n)}{\varphi(n)}$ is sub-mult.
\item [b)] $f(n)=\cfrac{\varphi(n)}{d(n)}$ is sup-mult.
\item [c)] $f(n)=n+d(n)$ is sub-mult.
\end{enumerate}
\end{ex}

\begin{theo}
Let $f$ be nonnegative sub-mult (sup-mult) function, and suppose that $f(n)\substack{\leq\\ (\geq)} n$ for all $n\geq 1.$ Then $f$ is sub-hom (sup-hom) (abbreviation for sub-homogeneous function).
\end{theo}

\begin{proof}
As $f(mn)\leq f(m)f(n)$ and $0\leq f(m)\leq m,$ we have $f(m)f(n)\leq mf(n),$ so $f(mn)\leq mf(n).$ The "sup-case" follows in the same manner.
\end{proof}

\begin{ex}
\begin{enumerate}
\item [a)] $f(n)=d(n)$ is sub-mult and $d(n)\leq n,$ so $d(n)$ is sub-hom.
\item [b)] $f(n)=ng(n),$ where $g(n)\geq 1$ is sup-hom. Then Th.1 a) and Th.2 implies that $f$ is sup-hom. For example, $f(n)=n\varphi(n)$ is sup-hom. On the other hand, $f(n)=\cfrac{n}{\varphi(n)}$ is sub-hom, as $f$ is sub-mult and $f(n)\leq n.$
\end{enumerate}
\end{ex}

\begin{theo}
Assume that $f\geq 0$ is sub-mult (sup-mult) and sup-hom (sub-hom). Then $f$ is $k$-sup-hom ($k$-sub-hom) for any $k\geq 2.$
\end{theo}

\begin{proof}
As $f(mn)\leq f(m)f(n)$ one gets by induction that $f(a_1a_2\ldots a_k)\leq f(a_1)f(a_2)\ldots f(a_k).$ Letting $a_1=a_2=\ldots = a_k=mn,$ we get $\big(f(mn)\big)^k\geq f\big((mn)^k\big)=f(m^kn^k).$ Since $f$ is sup-hom, by definition \eqref{eq:5} we have $f(m^kn^k)\geq m^kf(n^k).$ Thus, we get $\big(f(mn)\big)^k\geq m^kf(n^k),$ which by \eqref{eq:8} gives that $f$ is $k$-sup-hom.
\end{proof}

\begin{theo}
Assume that $f\geq 0$ is $k$-sub-mult ($k$-sup-mult) and that $f(n)\leq n$ $(f(n)\geq n) $ for all $n.$ Then $f$ is $k$-sub-hom ($k$-sup-hom).
\end{theo}

\begin{proof}
As by \eqref{eq:6} one has $\big(f(mn)\big)^k\leq f(m^k)f(n^k)$ and by $f(m^k)\leq m^k$ we get $\big(f(mn)\big)^k\leq m^kf(n^k),$ i.e. relation \eqref{eq:8}. The other case may be proved similarly.
\end{proof}

\begin{ex}
\begin{enumerate}
\item Let $f(n)=\varphi(n);$ which is 2-sub-mult. As $\varphi(n)\leq n,$ the function $\varphi$ is 2-sub-hom.
\item $f(n)=\sigma(n)$ is 2-sub-mult. Since $\sigma(n)\geq n,$ the function $\sigma$ is 2-sup-hom.
\item $f(n)=\cfrac{\sigma(n)}{d(n)}$ satisfies also the conditions of this theorem.
\end{enumerate}
\end{ex}

\begin{megj}
As $d(n)\geq n$ is not true, relation \eqref{eq:4} is not a consequence of this theorem.
\end{megj}

The next theorem involves also power functions.

\begin{theo}
Let $f,g>0$ and suppose that $f$ is sub-mult (sup-mult) and $g$ is sub-hom (sup-hom). Then the function $h(n)=f(n)^{g(n)/n}$ is sub-mult (sup-mult), too.
\end{theo}

\begin{proof}
\begin{align*}
& f(mn)^{g(mn)}\leq \big(f(m)f(n)\big)^{g(mn)}=\big(f(m)\big)^{g(mn)}\cdot \big(f(n)\big)^{g(mn)}\leq\\
& \leq \big(f(m)\big)^{ng(m)}\cdot \big(f(n)\big)^{mg(n)}=\left[f(m)^{g(m)}\right]^n\cdot \left[f(n)^{g(n)}\right]^m
\end{align*}

Therefore, $h(mn)=f(mn)^{g(mn)/mn}\leq f(m)^{g(m)/m}\cdot f(n)^{g(n)/n}=h(m)\cdot h(n).$
\end{proof}

\begin{cor}
Assume that for $f,g$ satisfying the conditions of Theorem 5 one has for any prime $p\geq 2$
\begin{equation}\label{eq:10}
f(p)^{g(p)}<p^p.
\end{equation}

Then one has the inequality
\begin{equation}\label{eq:11}
f(n)^{g(n)}<n^n
\end{equation}
for any $n>1.$
\end{cor}

\begin{proof}
Let $m,n$ two numbers satisfying \eqref{eq:11}. Then as $f(n)^{g(n)/n}<n$ and $f(m)^{g(m)/m}<m,$ by Theorem 5 we have $$f(mn)^{g(mn)/mn}\leq f(m)^{g(m)/m}\cdot f(n)^{g(n)/n}<mn, \mbox{ so } f(mn)^{g(mn)}<mn,$$ i.e. inequality \eqref{eq:11} holds true also for $mn.$

Thus, if \eqref{eq:10} is true, then applying the above remark for $m=p, n=p;$ we get that \eqref{eq:11} will be true for $n=p^2,$ too. By induction this implies that \eqref{eq:11} is true for any prime power $p^k.$ Similarly, it is true for any other prime power $q^s,$ so by the remark it will be true for $p^k\cdot q^k,$ too. By induction it follows the inequality for any product of distinct prime powers, i.e. for any $n.$
\end{proof}

\begin{ex}
\begin{enumerate}
\item [a)] Let $f(n)=\sigma(n),$ which is sub-mult, and $g(n)=\varphi(n),$ which is sub-hom. The inequality \eqref{eq:10} becomes
\begin{equation}\label{eq:12}
(p+1)^{p-1}<p^p.
\end{equation}
\eqref{eq:12} may be written also as $\left(\cfrac{p+1}{p}\right)^p<p+1$ or $\left(1+\cfrac{1}{p}\right)^p<p+1.$ As $\left(1+\cfrac{1}{p}\right)^p<e<3\leq p+1$ for $p\geq 2,$ this is true. Therefore,
\begin{equation}\label{eq:13}
\sigma(n)^{\varphi(n)}<n^n \mbox{ for all } n\geq 2.
\end{equation}
\end{enumerate}
\end{ex}

\begin{megj}
Inequality \eqref{eq:13} appeared for the first time in \cite{4}, \cite{5}. There are many other inequalities of these types.
\begin{enumerate}
\item [b)] Other examples for Theorem 5: $$f(n)=d(n), \ g(n)=\cfrac{\sigma(n)}{d(n)};$$ $$f(n)=\cfrac{\sigma(n)}{d(n)}; \ g(n)=\sigma(n),$$ etc.
\end{enumerate}
\end{megj}

Now we shall find conditions for the sub-multiplicativity, $k$-sub-multiplicativity, sub-homogeneity, etc. properties.

\begin{theo}
Let $f\geq 0$ be a multiplicative function, $f(1)=1,$ and assume that for any prime numbers $p$ and any $a,b\geq 0$ one has the inequality
\begin{equation}\label{eq:14}
f(p^{a+b})\substack{\leq\\ (\geq)} f(p^a)\cdot f(p^b).
\end{equation}
Then $f$ is sub-mult (sup-mult) function. Reciprocally, if $f$ is sub-mult (sup-mult), then \eqref{eq:14} is true.
\end{theo}

\begin{proof}
We may assume that two positive integers $m$ and $n$ has the same number of distinct prime factors, with exponents $\geq 0.$ Let $m=p_1^{a_1}\cdots p_r^{a_r},$ $n=p_1^{a'_1}\cdots p_r^{a'_r}$ with $p_1,\ldots p_r$ distinct primes, and $a_i\geq 0,$ $a'_i\geq 0.$ As $f$ is multiplicative, and $f(1)=1,$ clearly $f(m)=f(p_1^{a_1})\cdots f(p_r^{a_r}).$

Now, as by \eqref{eq:14} one has

\begin{equation}\label{eq:15}
f(p_i^{a_i+a'_i})\leq f(p_i^{a_i})\cdot f(p_i^{a'_i})
\end{equation}
after term-by term multiplication in \eqref{eq:15}, and using the multiplicativity of $f,$ we get
$$f(mn)\leq f(m)\cdot f(n),$$ i.e. $f$ is sub-multiplicative. Reciprocally, letting $m=p^a,$ $n=p^b,$ we get relation \eqref{eq:14}.

\begin{ex}
\begin{enumerate}
  \item [a)] Let $f(n)=d(n),$ multiplicative. As $d(p^a)=a+1$ for any prime $p$ and any $a\geq 0$ (as $d(1)=1$), \eqref{eq:14} becomes $a+b+1\leq (a+1)(b+1)=ab+a+b+1,$ i.e. $ab\geq 0;$ which is trivial. Thus, the function $d$ is sub-mult.

      A similar case is $f(n)=\sigma(n)$ (with a slightly more involved proof)
  \item [b)] Let $f(n)=\cfrac{\sigma(n)}{d(n)},$ which is multiplicative. When $a=0$ or $b=0,$ \eqref{eq:14} is trivial, so we may assume $a,b\geq 1.$ We will prove the inequality
\begin{equation}\label{eq:16}
\frac{p^{a+b+1}-1}{(p-1)(a+b+1)}\geq\frac{p^{a+1}-1}{(p-1)(a+1)}\cdot \frac{p^{b+1}-1}{(p-1)(b+1)}.
\end{equation}
\end{enumerate}
\end{ex}

For this reason, apply the Chebyshev integral inequality (see \cite{2})
\begin{equation}\label{eq:17}
\frac{\int_\alpha^\beta f(x)\mbox{d}x}{\beta-\alpha}\cdot \frac{\int_\alpha^\beta g(x)\mbox{d}x}{\beta-\alpha}\leq \frac{\int_\alpha^\beta f(x)g(x)\mbox{d}x}{\beta-\alpha}
\end{equation}
to the case $f(x)=x^a,$ $g(x)=x^b,$ $\alpha=1,$ $\beta=p.$ Since $f$ and $g$ are monotonic functions of the same type, \eqref{eq:17} holds true, so \eqref{eq:16} follows.

This proves that the function $f(n)=\cfrac{\sigma(n)}{d(n)}$ is sup-mult (but $\sigma$ and $d$ are both sub-mult).
\end{proof}

\begin{theo}
Let $f\geq 0$ be multiplicative, $f(1)=1,$ and assume that for any primes $p,$ and any $a,b\geq 0$ one has
\begin{equation}\label{eq:18}
\big(f(p^{a+b})\big)^k\substack{\leq\\ (\geq)}f(p^{ka})\cdot f(p^{kb})
\end{equation}
where $k\geq 2$ is fixed. Then $f$ is $k$-sub-mult ($k$-sup-mult) function.
\end{theo}

\begin{proof}
This is very similar to the proof of Theorem 6, and we shall omit it.
\end{proof}

\begin{ex}
\mbox{}

\begin{enumerate}
  \item [a)] Let $f(n)=d(n),$ $k\geq 2.$ We will prove that $f$ is $k$-sup-mult (though, it is sub-mult, see Example 5 a)), i.e. the inequality
\begin{equation}\label{eq:19}
\big(d(p^{a+b})\big)^k \geq d(p^{ka})\cdot d(p^{kb}),
\end{equation}
i.e.
\begin{equation}\label{eq:20}
(a+b+1)^k\geq (ka+1)(kb+1).
\end{equation}

We shall prove inequality \eqref{eq:20} by induction upon $k.$ For $k=2$ it is true, as $(a+b+1)^2=a^2+b^2+1+2ab+2a+2b\geq (2a+1)(2b+1)=4ab+2a+2b+1,$ i.e. $a^2+b^2\geq 2ab,$ which is true by $(a-b)^2\geq 0.$

Assume that \eqref{eq:20} holds true for $k,$ and try to prove it for $k+1:$ $(a+b+1)^{k+1}\geq (ka+1)(kb+1)(a+b+1)\geq (ka+a+1)(kb+b+1).$

After some easy computations this becomes
$$k^2a^2b+k^2ab^2+ka^2+kb^2\geq ab,$$ which is true. Thus \eqref{eq:19} holds for any $k\geq 2.$
\end{enumerate}

\begin{megj}
For $k=2,$ the 2-sup-mult property of $d$ was first published in \cite{4}. For a recent rediscovery of this result, see  \cite{8}(Theorem 2.3).
\end{megj}

\begin{enumerate}
\item [b)] Let $f(n)=\varphi(n), k\geq 2.$ We will prove that $f$ is $k$-sub-mult, i.e.; \eqref{eq:18} holds true. We may assume $a,b\geq 1.$ As
$\varphi(p^m)=p^{m-1}\cdot (p-1)$ for $p$ prime, $m\geq 1;$ we have to show that
$$p^{k(a+b-1)}(p-1)^k\leq p^{ka-1}(p-1)p^{kb-1}(p-1).$$

This becomes
$$\left(1-\frac{1}{p}\right)^k\leq \left(1-\frac{1}{p}\right)^2,$$
which is true for any $k\geq 2,$ as $0<1-\cfrac{1}{p}<1.$
\end{enumerate}

\begin{megj}
For $k=2,$ the 2-sub-mult property of $\varphi$ was first discovered by T. Popoviciu \cite{3}.
\end{megj}
\end{ex}

\begin{theo}
Let $f\geq 0$ be multiplicative, $f(1)=1,$ and assume that
\begin{equation}\label{eq:21}
f(p^{a+b})\substack{\leq\\(\geq)} p^a\cdot f(p^b)
\end{equation}
holds true for any primes $p\geq 2$ and any $a,b\geq 0.$ Then $f$ is sub-hom (sup-hom) function.
\end{theo}

\begin{proof}
Similar to the proof of Theorem 6.
\end{proof}

\begin{ex}
\begin{enumerate}
  \item [a)] Let $f(n)=\sigma(n).$ We will prove \eqref{eq:21} with "$\geq$" sign of inequality. If $b=0,$ $\sigma(p^a)\geq p^a$ is true. For $a=0$ we get equality. Assume $a,b\geq 1.$ Then \eqref{eq:21} becomes $1+p+\cdots + p^{a+b}\geq p^a(1+p+\cdots +p^b),$ which is true, as the left-hand side contains also the terms $p^a+p^{a+1}+\cdots + p^{a+b},$ which is the right side.
  \item [b)] For $f(n)=\varphi(n),$ we will prove \eqref{eq:21} with "$\leq$" sign (i.e. $\varphi$ will be sub-hom). As $\varphi(p^a)\leq p^a,$ \eqref{eq:21} holds for $b=0,$ while for $a=0$ one has equality. For $a,b\geq 1$ we have to prove
      $$p^{a+b-1}(p-1)\leq p^a\cdot p^{b-1}\cdot (p-1),$$ which holds true with equality.
\end{enumerate}
\end{ex}

\begin{theo}
Let $f\geq 0$ be multiplicative, $f(1)=1$ and assume that for any primes $p$ and any $a,b\geq 0$ one has
\begin{equation}\label{eq:22}
\big(f(p^{a+b})\big)^k\substack{\leq\\(\geq)} p^{ka}\cdot f(p^{kb}).
\end{equation}

Then $f$ is $k$-sub-hom ($k$sup-hom) function.
\end{theo}

\begin{proof}
This is similar to the proof of Theorem 6.
\end{proof}

\begin{ex}
\begin{enumerate}
  \item [a)] Let $f(n)=\varphi(n).$ We will prove that $f$ is $k$-sub-hom. Inequality
\begin{equation}\label{eq:23}
\big(\varphi(p^{a+b})\big)^k\leq p^{ka}\cdot \varphi(p^{kb})
\end{equation}
holds true for $b=0,$ as $\varphi(p^a)\leq p^a.$ For $a=0$ the inequality $\big(\varphi(p^b)\big)^k\leq\varphi(p^{kb})$ holds true by $\varphi(\underbrace{p^b\cdots p^b}_{k})\geq\varphi(p^b)\cdots \varphi(p^b),$ using the sup-mult property of $\varphi.$ Let $a,b\geq 1.$ Then \eqref{eq:23} becomes
$$p^{k(a+b-1)}\cdot (p-1)^k\leq p^{ka}\cdot p^{kb-1}\cdot (p-1),$$ i.e.$\left(1-\cfrac{1}{p}\right)\leq 1-\cfrac{1}{p}.$ As $0<1-\cfrac{1}{p}<1$ and $k>1,$ this is true.
 \item [b)] $f(n)=\sigma(n)$ is $k$-sup-hom; and follows similarly from \eqref{eq:22}. We omit the details.

\end{enumerate}
\end{ex}

\end{document}